\documentclass[a4paper]{article}
\pdfoutput=1
\usepackage{amsmath}

\usepackage{amsthm}
\usepackage{graphicx}

\usepackage{amscd}
\usepackage{epsfig}
\usepackage{color}

\usepackage{psfrag}
\usepackage{srcltx}

\usepackage{amssymb}
\usepackage{latexsym}
\usepackage{tikz-cd}

\makeatletter
\newcommand\footnoteref[1]{\protected@xdef\@thefnmark{\ref{#1}}\@footnotemark}
\makeatother

\newtheorem{Theorem}{Theorem}

\newtheorem{Lemma}[Theorem]{Lemma}
\newtheorem{Corollary}[Theorem]{Corollary}

\theoremstyle{definition}

\def\Z{\mathbb Z}

\begin{document}

\title{Non-isometric hyperbolic $3$-orbifolds with the same topological type 
and volume}

\author{J\'er\^ome Los,
Luisa Paoluzzi,
and Ant\'onio Salgueiro\footnote{Partially supported by the Centre for 
Mathematics of the University of Coimbra -- UID/MAT/00324/2013, funded by the 
Portuguese Government through FCT/MEC and co-funded by the European Regional 
Development Fund through the Partnership Agreement PT2020.}
} 
\date{\today}


\maketitle


\begin{abstract}
\vskip 2mm

We construct pairs of non-isometric hyperbolic $3$-orbifolds with the same 
topological type and volume. Topologically these orbifolds are mapping tori of 
pseudo-Anosov maps of the surface of genus $2$, with singular locus a fibred 
(hyperbolic) link with five components.

\vskip 2mm


\vskip 2mm


\end{abstract}

\section{Introduction}
\label{s:introduction}

In this brief note, building upon techniques and ideas already used in 
\cite{LPS}, we construct a hyperbolic $3$-manifold that is a branched cover of
a link in another hyperbolic manifold (the mapping torus of an pseudo-Anosov 
map of the surface of genus $2$) in two different ways. More precisely we show

\begin{Theorem}
Given two integers $n>m\ge2$, there are infinitely many non isometric pairs
$(M,L)$ where $M$ is a fibred hyperbolic $3$-manifold and $L$ a five-component 
link contained in another fibred hyperbolic $3$-manifold $|O|$ such that $M$ is 
a $2mn$-sheeted branched cover of $L$ in two non-equivalent ways. 
\end{Theorem}

Here $L$ is transverse to the fibration of the manifold in which it is 
contained.

As a straightforward consequence of the above result, we obtain 

\begin{Corollary}
There are infinitely many pairs of non-isometric hyperbolic $3$-orbifolds that 
have the same topological type and volume. 
\end{Corollary}

Indeed, the two orbifolds of each pair mentioned in the corollary are obtained 
as quotients of $M$ by the action of two groups of order $2mn$: in both cases, 
the space of orbits is $|O|$ and the points with non trivial stabilisers map 
onto $L$ with orders of ramification $(2,2,2m,2m,n)$ in one case and 
$(2,2,2n,2n,m)$ in the other.

In order to build our examples, we start by describing certain (branched) 
covers of surfaces. As in \cite{LPS}, the $3$-manifolds and orbifolds we are
looking for will be then obtained as mapping tori of pseudo-Anosov maps 
defined on the surfaces considered. The maps will be chosen so that they
``commute" with the covering projections, indeed all maps will be lifts
of a single Anosov map defined on the common quotient of all surface covers,
that is a torus in this construction.

\section{Surface covers}\label{s:surfcov}

Given two integers $n>m\ge2$ we wish to construct two non-equivalent
$2mn$-sheeted covers from the (closed, connected, orientable) surface $S_g$ of 
genus $g=6nm-n-m+1$ onto the surface of genus $2$, branched over five points, 
with orders of ramification $(2,2,2m,2m,n)$ for the first cover and 
$(2,2,2n,2n,m)$ for the second one. We shall see the bases of the two covers as 
two $2$-orbifolds: $\Sigma_2(2,2,2m,2m,n)$ and $\Sigma_2(2,2,2n,2n,m)$. Both 
orbifolds are orbifold double covers of the torus $T(2,2,2m,2n)$ with four cone 
points of orders $(2,2,2m,2n)$. 

In order to make the construction of the different coverings easier to 
understand we will separate it in two different steps. In step one we will
construct an $mn$-sheeted branched cover from $S_g$ to the orbifold
$\Sigma_5(m,m,n,n)$ of genus five with four cone points of orders $(m,m,n,n)$. 
In step two we will show that the orbifold $\Sigma_5(m,m,n,n)$ double covers 
both $\Sigma_2(2,2,2m,2m,n)$ and $\Sigma_2(2,2,m,2n,2n)$. Using the fact that 
the deck transformations of the covering constructed commute, we will also see 
that $\Sigma_2(2,2,2m,2m,n)$ and $\Sigma_2(2,2,m,2n,2n)$ have a common 
quotient, which is the $\Z/2\times\Z/2$ quotient of $\Sigma_5(m,m,n,n)$. This 
quotient will be $T(2,2,2m,2n)$.

\begin{figure}[h]
\begin{center}
 {
  \includegraphics[height=2cm]{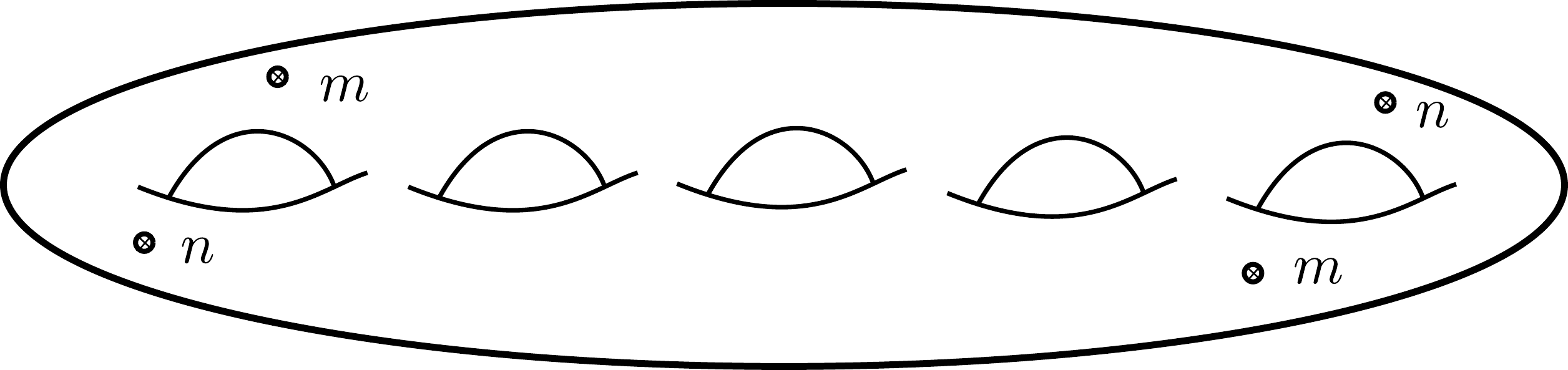}
 }
\end{center}
\caption{The orbifold $\Sigma_5(m,m,n,n)$.}
\label{fig2}
\end{figure}

The branched covers just discussed can be summarised in the following commuting
diagram of covers:

\begin{center}
\begin{tikzcd}[column sep=-10pt]
& S_g \arrow{d} & \\
&  \Sigma_5(m,m,n,n) \arrow{dl} \arrow{dr}&\\
\Sigma_2(2,2,2m,2m,n) \arrow{dr} && \Sigma_2(2,2,m,2n,2n) \arrow{dl}\\
& T(2,2,2m,2n) &\\
\end{tikzcd}
\end{center}

Note that all covers associated to arrows appearing in the diagram are regular.
The group of deck transformations associated to $S_g\longrightarrow
\Sigma_5(m,m,n,n)$ is isomorphic to $\Z/m\times\Z/n$, the other being
double covers, as already observed.

\subsection{Step one}

In this part we start by constructing a cover of order $mn$ from the
surface of genus $(m-1)(n-1)$ onto the $2$-sphere with four branch points, two
of order $n$ and two of order $m$. Having fixed an integer $h\ge 0$ we will
then adapt the construction in order to have a cover from the surface of genus 
$g=hmn+(m-1)(n-1)$ onto the surface of genus $h$ again with four branch points,
two of order $n$ and two of order $m$. Observe that for the case we are 
interested in we have $h=5$ and $g=6nm-n-m+1=5mn+(m-1)(n-1)$.

To construct the covering we want, it is sufficient to find a surface of genus
$(m-1)(n-1)$ admitting a symmetry of type $\Z/m\times\Z/n$ where the generators
of both cyclic subgroups have fixed points belonging to precisely two orbits of
the $\Z/m\times\Z/n$-action. A simple way to build a surface having prescribed 
symmetry is to use a symmetric graph and see the 
surface as the boundary of a regular neighbourhood of a standard embedding of 
the graph in $3$-space (as was done in \cite{LPS}).

We will build a graph embedded in the $3$-sphere. To make things more
explicit, it is convenient to see the $3$-sphere ${\mathbf S}^3\subset
{\mathbb C}^2$ as the set  of points $(z_1,z_2)$ such that $|z_1|^2+|z_2|^2=1$. 
Note that ${\mathbf S}^3$ admits a $\Z/m\times\Z/n$-action defined on the
generators as $(z_1,z_2)\mapsto (e^{2i\pi/m}z_1,z_2)$ and 
$(z_1,z_2)\mapsto (z_1,e^{2i\pi/n}z_2)$. Consider now the following sets of
points in $3$-sphere: $A=\{p_k=(e^{2ik\pi/m},0) \mid k=0,\dots, m-1\}$ and
$B=\{q_l=(0,e^{2il\pi/n})\mid l=0,\dots, n-1\}$. Observe that both sets are
invariant by the $\Z/m\times\Z/n$-action by construction. The graph we are
interested in is the complete bipartite graph with sets of vertices $A$ and $B$ and set of edges $\{e_{kl}:k=0,\ldots,m-1;l=0,\ldots,n-1\}$ where 
$e_{kl}=\{(\cos t\,p_k+\sin t\,q_l):0<t<\pi/2\}$. It is clear that the graph is 
embedded in ${\mathbf S}^3$ in a $\Z/m\times\Z/n$-equivariant way. 
Figure~\ref{fig_tori} shows the case $n=3$ and $m=4$, where the solid tori 
$\{(z_1,z_2)\in{\mathbf S}^3:|z_2|^2\leq \frac12\}$ and 
$\{(z_1,z_2)\in{\mathbf S}^3:|z_2|^2\geq \frac12\}$ have their boundaries identified, 
and their cores are the components of a Hopf link. Figure~\ref{fig1} shows the 
whole graph for $n=2$ and $m=3$. 
\begin{figure}[h]
\begin{center}
 {
  \includegraphics[height=6cm]{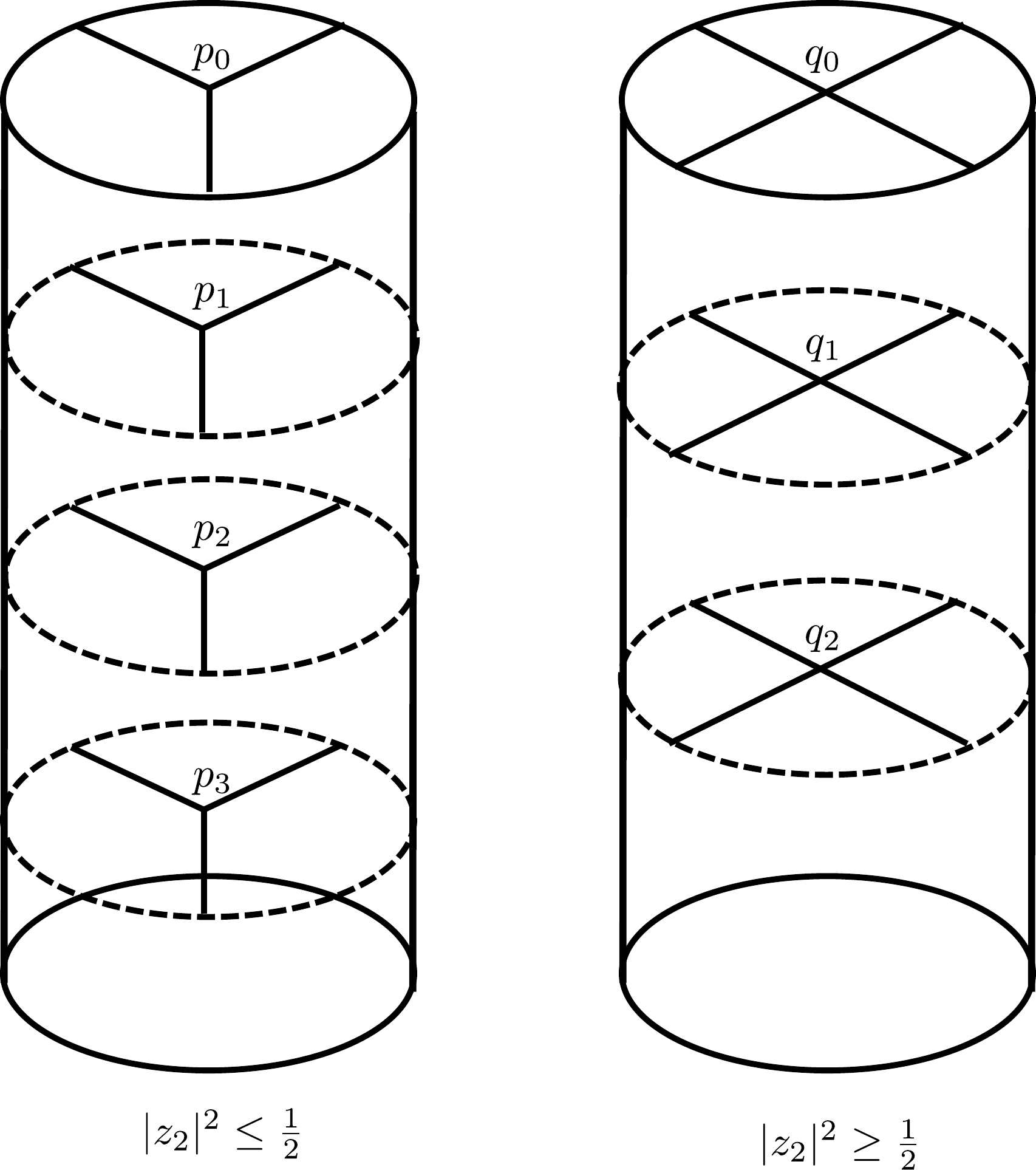}
 }
\end{center}
\caption{A local picture of the $(n,m)$-complete bipartite graph with vertices 
on the Hopf link for $n=3$ and $m=4$.}
\label{fig_tori}
\end{figure}

\begin{figure}[h]
\begin{center}
 {
  \includegraphics[height=4cm]{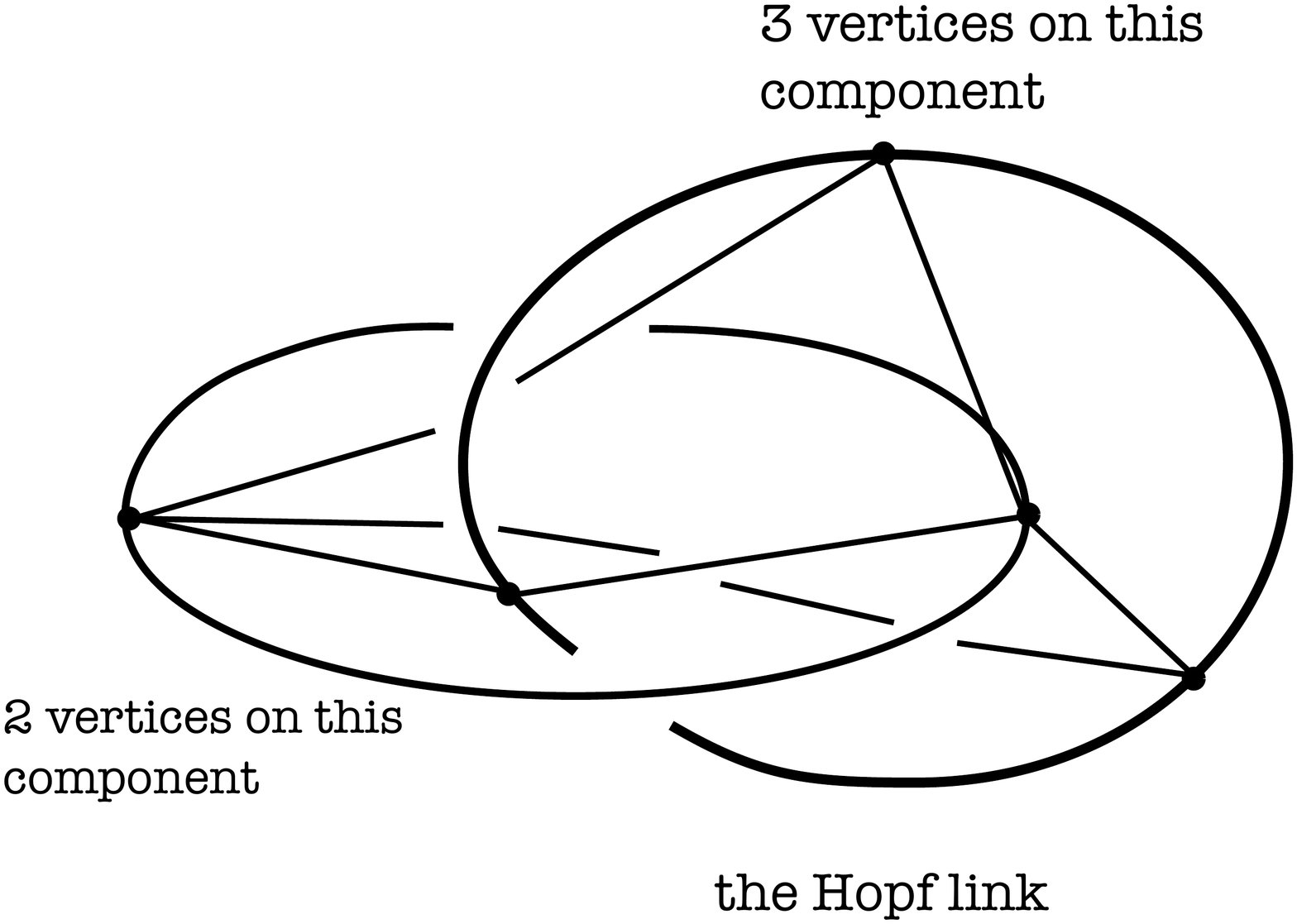}
 }
\end{center}
\caption{A global picture of the $(n,m)$-complete bipartite graph with vertices 
on the Hopf link for $n=2$ and $m=3$.}
\label{fig1}
\end{figure}

We then obtain the 
desired surface by taking the boundary of a sufficiently small regular and 
equivariant neighbourhood of the graph. Note that the Euler characteristic of 
our graph is $m+n-nm$, so it follows immediately that the genus of the boundary 
surface is $mn-m-n+1$. We wish to stress that the fixed-point set of the cyclic 
subgroup of order $m$ acting on the $3$-sphere is the circle of equation 
$z_1=0$ containing the set $B$ and, similarly, the fixed-point set of the 
cyclic subgroup of order $n$ acting is the circle of equation $z_2=0$ 
containing the set $A$. As a consequence, nearby each point of $A$ 
(respectively $B$) the cyclic group of order $n$ (respectively $m$) has two 
fixed points on the surface. Since the $\Z/m\times\Z/n$-action is freely 
transitive on the edges of the graph, it is not hard to see that the quotient 
of this surface by the action is a sphere with two cone points of order $m$ and 
two of order $n$. 

\begin{figure}[h]
\begin{center}
 {
  \includegraphics[height=5cm]{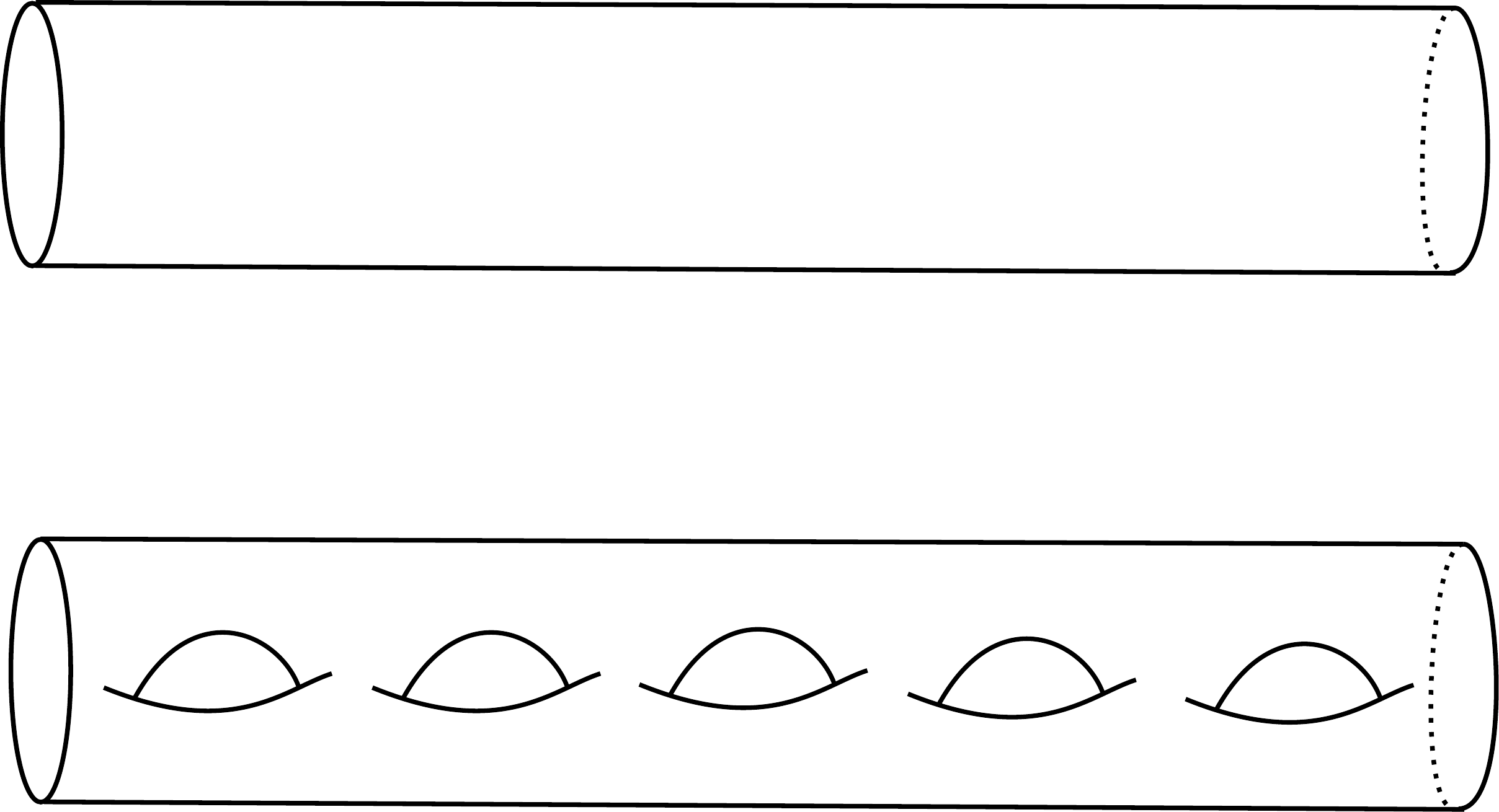}
 }
\end{center}
\caption{A tube and a surface of genus $h=5$ that replaces it in the
construction.}
\label{fig4}
\end{figure}

To construct a covering with the same type of action from the surface of genus
$hmn+(m-1)(n-1)$ onto a surface of genus $h$, we start by observing that the
surface we have just built can be decomposed into $n$ $m$-holed spheres, $m$
$n$-holed spheres and $mn$ tubes each attached on one side to some hole of a 
sphere of the first type and on the other to some hole of a sphere of the 
second type. To generalise the construction it suffices to replace each tube 
with a surface of genus $h$ with two holes (see Figure~\ref{fig4}): the 
boundary components of the surface are attached as the boundary components of 
the original tubes were. It is obvious that one can carry out this construction 
in a $\Z/m\times\Z/n$-equivariant way.

\subsection{Step two}

\begin{figure}[h]
\begin{center}
 {
  \includegraphics[height=6cm]{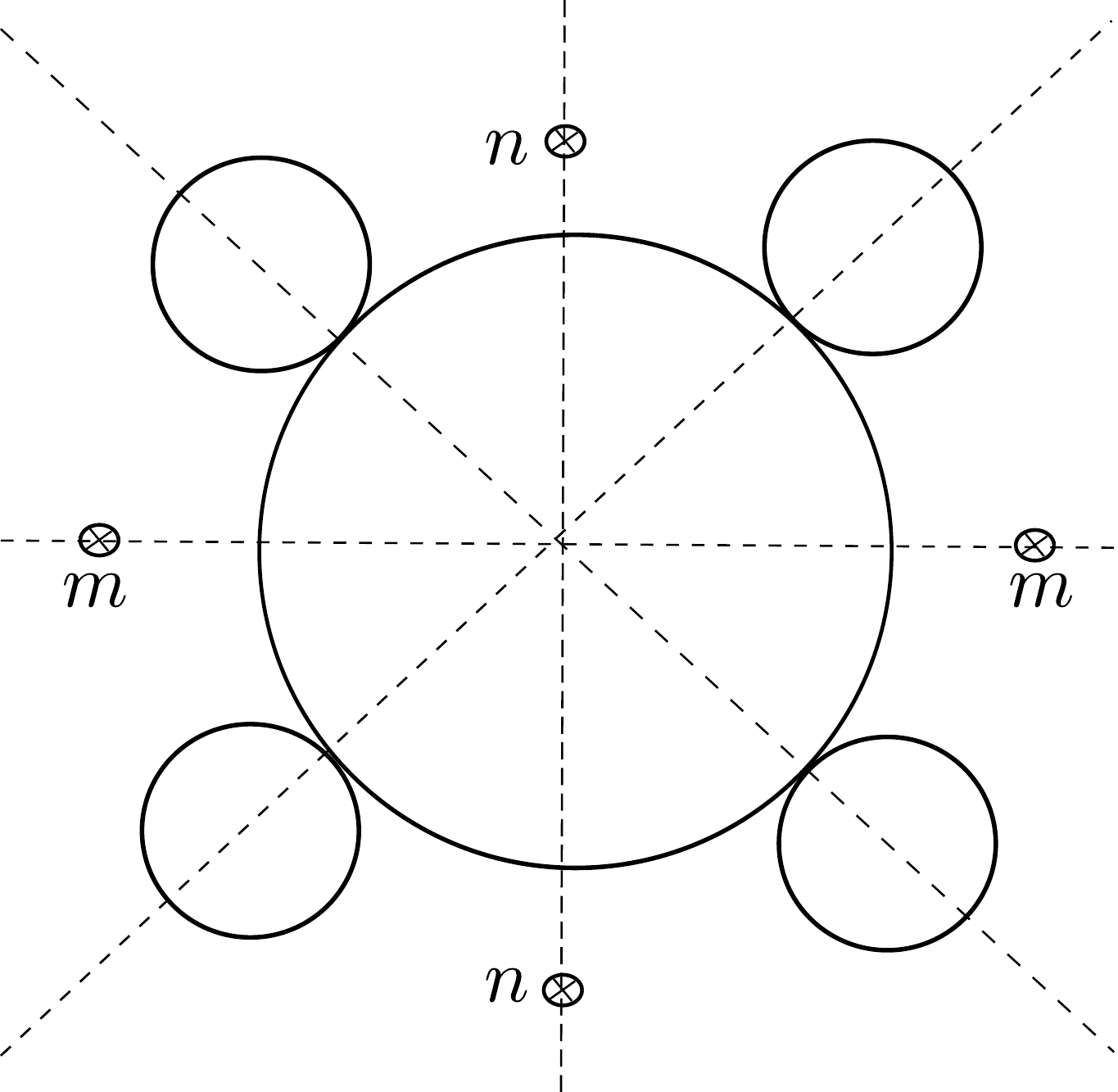}
 }
\end{center}
\caption{The circle pattern in the $x_3=0$-plane and the position of the cone
points of orders $n$ and $m$ on the genus-$5$ surface boundary of a regular
neighbourhood (not shown in the picture).}
\label{fig3}
\end{figure}

For the remaining coverings we follow the same strategy: we start by
considering a symmetric graph embedded in ${\mathbb R}^3$. Consider the unit
circle $C_0$ in the plane of equation $x_3=0$. In the same plane, consider four 
circles $C_i$, $i=1,2,3,4$, of centres $(1,1,0)$, $(-1,1,0)$, $(-1,-1,0)$, and 
$(1,-1,0)$ respectively, all of radius $\sqrt{2}-1$. The set 
$\Gamma=\cup_{i=0}^4 C_i$ is a graph in the plane of equation $x_3=0$. By
construction, $\Gamma$ is invariant by the action of the $\pi$-rotations about
the $x_1$ and $x_2$ axes. Once more, consider a small and invariant regular
neighbourhood of $\Gamma$: its boundary is clearly a surface of genus $5$. We
wish to identify this surface with the quotient $\Sigma_5(m,m,n,n)$ of the
cover constructed in Step one. We do so by imposing that the outermost points
of intersection of the surface with the $x_1$ axis are the two cone points of 
order $m$ and the outermost points of intersection of the surface with the 
$x_2$ axis are two cone points of order $n$. It is now straightforward to
realise that the quotient of $\Sigma_5(m,m,n,n)$ by the $\pi$-rotation about
the $x_1$ (respectively $x_2$) axis is the orbifold $\Sigma_2(2,2,2m,2m,n)$
(respectively $\Sigma_2(2,2,2n,2n,m)$). The two $\pi$-rotations commute and
generate a Klein group of order four. The quotient of $\Sigma_5(m,m,n,n)$ by 
its action is $T(2,2,2m,2n)$. 

Remark that, although the two double branched covers from the surface of genus
five to that of genus two are equivalent, the $2mn$-sheeted branched coverings 
$S_g\longrightarrow \Sigma_2(2,2,2m,2m,n)$ and $S_g\longrightarrow
\Sigma_2(2,2,2n,2n,m)$ are not, for the orders of ramification are different.
Notice that the two double branched covers are equivalent since the surface of 
genus five admits a cyclic symmetry of order four (a $\pi/2$-rotation about the 
$x_3$-axis) which conjugates the two covering involutions. This symmetry 
preserves the marked points, although not their orders and generates, together 
with the Klein group, a dihedral group of order eight.

\section{Mapping tori}

We wish now to construct $3$-manifolds as mapping tori of homeomorphisms
defined on the surfaces built in the previous section. We want, moreover,
that the mapping tori fulfill some extra requirements.

\begin{Lemma}
Consider the five surfaces $S_g$, $\Sigma_5(m,m,n,n)$, $\Sigma_2(2,2,2m,2m,n)$, 
$\Sigma_2(2,2,2n,2n,m)$, and $T(2,2,2m,2n)$ constructed in the previous
section. For each of the five surfaces it is possible to choose a homeomorphism
of the surface in such a way that the five chosen homeomorphisms and the
associated mapping tori satisfy the following conditions. 
\begin{enumerate}
\item Each covering of surfaces induces a covering of the corresponding mapping
tori.
\item The mapping tori associated to the homeomorphisms of 
$\Sigma_2(2,2,2m,2m,n)$ and $\Sigma_2(2,2,2n,2n,m)$ are homeomorphic.
\item The cone points of $\Sigma_2(2,2,2m,2m,n)$ and $\Sigma_2(2,2,2n,2n,m)$ 
are fixed by the chosen homeomorphisms and they close up to equivalent
five-component links in the mapping tori.
\end{enumerate} 
\end{Lemma}

\begin{proof}
The first condition is easy to fulfill. Indeed, it is sufficient to choose any
homeomorphism $\psi$ of the torus which fixes four points (corresponding to the 
four cone points of $T(2,2,2m,2m)$). Since we are dealing with finite covers 
there is a non-trivial power of $\psi$ that lifts to all covers and, up to 
choosing possibly a further power, we can even ensure that the fibres of the 
cone points are fixed pointwise by the lift. 

To ensure that the remaining two conditions hold, we need that the lifts of
$\psi$ to $\Sigma_2(2,2,2m,2m,n)$ and $\Sigma_2(2,2,m,2n,2n)$ are conjugate, 
and the conjugation maps the cone points of the first orbifold to those of the
second. 

To achieve this we need to start with a further quotient of our surfaces. We
point out that this quotient will not be the base space of an orbifold covering
with total space $S_g$. Consider $T(2,2,2m,2m)$. If we forget the orders of the
cone points we see that the symmetry of order four of the surface of genus five
induces an elliptic involution of $T(2,2,2m,2m)$ which exchanges the two cone
points of order two and those of orders $2m$ and $2n$. Observe that this
elliptic involution is also induced by any of the two $\pi$-rotations about the
diagonals $x_1=\pm x_2$ in the $x_3=0$ plane of the previous construction (see
Figure~\ref{fig3}).  

We can thus start with a homeomorphism $\varphi$ of the $2$-sphere quotient of 
the torus by the action of the elliptic involution. We require that $\varphi$ 
fixes six points: two points $a$ and $b$ corresponding to the orbits of the 
cone points of $T(2,2,2m,2m)$ and four others corresponding to the orbits of 
the fixed-points of the elliptic involution. 
Up to passing to a power, we can assume that $\varphi$ lifts to a homeomorphism
$\psi$ of the covering torus which fixes each point $a_1,a_2$ in the fibre of 
$a$ and $b_1,b_2$ in the fibre of $b$. We identify this torus with 
$T(2,2,2m,2n)$ in such a way that $a_1$ and $a_2$ are cone points of order $2$, 
$b_1$ of order $2m$ and $b_2$ of order $2n$. We also require that $\{a_i,b_i\}$ 
are the images of the fixed points of the $\pi$-rotation about the axis $x_i$, 
$i=1,2$. 

With this choice of $\psi$, the mapping tori associated to the lifts of $\psi$
discussed above satisfy all given requirements.
\end{proof}

The diagram of surface orbifold covers give thus rise to a diagram of induced
$3$-dimensional orbifold covers between mapping tori:

\begin{center}
\begin{tikzcd}[column sep=-10pt]
& M \arrow{d} & \\
&  N(m,m,n,n) \arrow{dl} \arrow{dr}&\\
O(2,2,n,2m,2m) \arrow{dr} && O(2,2,m,2n,2n) \arrow{dl}\\
& Q(2,2,2m,2m) &\\
\end{tikzcd}
\end{center}

We stress again that here the orbifolds $O(2,2,n,2m,2m)$ and $O(2,2,2n,2n,m)$
have the same topological type, that is are homeomorphic as manifolds and their
singular sets are equivalent links. 

\section{Hyperbolic structures}

To ensure that our mapping tori are hyperbolic manifolds (and orbifolds) it
suffices to choose the homeomorphism $\psi$ described in the previous
section to be Anosov. This can be done without considering the quotient by the
elliptic involution. Indeed, each linear Anosov map commutes with the standard
elliptic involution corresponding to minus the identity; in particular, the
elliptic involution preserves the fixed points of the Anosov map. Now, to our 
purposes it is sufficient that the linear Anosov we choose has at least four 
fixed points that are exchanged in pairs by the hyperelliptic involution. 
Indeed, we only need to identify these two orbits with the set of cone points 
of order two, and of order $2m$ and $2n$ respectively. Since Anosov maps have 
periodic orbits of arbitrarily large orders, up to perhaps taking a power, we 
can assume that our linear Anosov has at least eight fixed points. Now the 
standard elliptic involution has precisely four fixed points, all other points
belonging to orbits with two elements. We can thus conclude that the Anosov map
has at least four points that are exchanged in pairs by the standard elliptic
involution.

Note that all finite group of deck transformations can be realised as isometry
groups for the hyperbolic structure of the $3$-manifolds or orbifolds on which
they act. Here again, since the orders of the groups acting are arbitrarily
large we see that we must have infinitely many non isometric manifolds $M$
having two non equivalent branched covers on the same link.

Also, by taking powers of $\psi$ we obtain infinitely many commensurable
examples and infinitely many links. Indeed, for a fixed choice of $n>m\ge2$,
and for a fixed link $L$, the set of possible volumes of orbifolds with 
singular set $L$ and ramifications orders of the form $(2,2,2n,2n,m)$ or 
$(2,2,2m,2m,n)$ must be finite. However, by taking mapping tori associated to
powers of $\psi$ the volume must grow, because they cover each other.  

Note that the five-component links in $|O|$ are fibred and hyperbolic. Indeed,
the pseudo-Anosov maps defined over the surface of genus $2$ with five cone
points restrict to pseudo-Anosov maps of the five-punctured genus-$2$ surface
obtained by removing the cone points.  

\begin{footnotesize}

\textsc{Aix-Marseille Univ, CNRS, Centrale Marseille, I2M, UMR 7373,
13453 Marseille, France}

{jerome.los@univ-amu.fr}

\textsc{Aix-Marseille Univ, CNRS, Centrale Marseille, I2M, UMR 7373,
13453 Marseille, France}

{luisa.paoluzzi@univ-amu.fr}

\textsc{Department of Mathematics, University of Coimbra, 3001-454 Coimbra, Portugal}

{ams@mat.uc.pt}

\end{footnotesize}

\end{document}